\documentclass[a4paper,10pt]{article}
\usepackage[latin1]{inputenc}
\usepackage[english]{babel}
\usepackage{amsthm}
\usepackage{amssymb}
\usepackage{amsmath}
\usepackage{hyperref}
\usepackage{graphicx}
\usepackage[all]{xy}
\DeclareMathAlphabet{\mathpzc}{OT1}{pzc}{m}{it}
\newtheoremstyle{note}{11pt}{11pt}{}{}{\bfseries}{.}{.5em}{}
\newtheorem{theo}[equation]{Theorem}

\newtheorem{prop}[equation]{Proposition}

\newtheorem{defin}[equation]{Definition}

\newtheorem{conj}[equation]{Conjecture}

\newtheorem{rem}[equation]{Remark}

\numberwithin{equation}{section}
\newtheorem{lemma}[equation]{Lemma}

\newcommand{\Qb}{\overline{\mathbb{Q}}}
\newcommand{\Q}{\mathbb{Q}}

\newcommand{\Z}{\mathbb{Z}}

\newcommand{\R}{\mathbb{R}}
\newcommand{\C}{\mathbb{C}}

\newcommand{\X}{\mathcal{X}}

\newcommand{\Ll}{\mathcal{L}}

\newcommand{\oo}{\mathcal{O}}
\newcommand{\h}{\mathcal{H}}
\newcommand{\rr}{\mathcal{R}}

\newcommand{\W}{\mathcal{W}}
\newcommand{\U}{\mathcal{U}}

\newcommand{\G}{\Gamma}

\newcommand{\eps}{\varepsilon}

\newcommand{\cycar}{{\chi_{\mathrm{cycl}}}}

\newcommand{\st}{\mathrm{st}}
\newcommand{\cris}{\mathrm{cris}}
\newcommand{\spin}{\mathrm{spin}}
\newcommand{\Ad}{\mathrm{Ad}}
\newcommand{\NN}{\mathrm{N}}

\newcommand{\boldD}{\mathbf{D}}
\newcommand{\boldB}{\mathbf{B}}

\newcommand{\pfrak}{\mathfrak{p}}

\newcommand{\lgr}{\left\{}
\newcommand{\rgr}{\right\}}

\newcommand{\Addresses}{{
  \bigskip
  \footnotesize

  G.~Rosso,\\
  
\textsc{Department of Mathematics, KU Leuven\\
Celestijneenlan 200B,\\
3001 Heverlee, Belgium\\}

\textsc{Department of Mathematics, Columbia University\\
2990 Broadway, \\
New York, NY 10027\\}

  {giovanni.rosso@wis.kuleuven.be}\\

}}

\makeindex
\begin{document}
\title{\texorpdfstring{$\mathcal{L}$-invariant for Siegel-Hilbert forms}{L-invariant for Siegel-Hilbert forms}}
\author{Giovanni Rosso 
\footnote{PhD Fellowship of Fund for Scientific Research - Flanders, partially supported by a FWO travel  grant (V4.260.14N) and an ANR grant  (ANR-10-BLANC 0114 ArShiFo), 
}}
\maketitle
We prove  in some cases a formula for the Greenberg-Benois $\Ll$-invariant of the spin, standard and adjoint Galois representations associated with Siegel-Hilbert modular forms. In order to simplify the calculation, we give a new definition of the $\Ll$-invariant for a Galois representation $V$ of a number field $F\neq \Q$; we also check that it is compatible with Benois' definition for $\mathrm{Ind}_F^{\Q}(V)$.
\tableofcontents
\section{Introduction}

Since the historical results of Kummer and Kubota-Leopold on congruences for Bernoulli numbers, people have been interested in studying the $p$-adic variation of special values of $L$-functions.\\
More precisely, fix a motive $M$ over $\Q$. We suppose that $M$ is Deligne critical at $s=0$ and that there exists a Deligne's period $\Omega(M)$ such that $\frac{L(M,0)}{\Omega(M)}$ is algebraic. Fix a prime $p$ and two embeddings 
\begin{align*}
\C_p \hookleftarrow \Qb \hookrightarrow \C.
\end{align*}
Let $V$ be the $p$-adic realization of $M$ and suppose that $V$ is semistable (\`a la Fontaine). Thanks to work of Coates and Perrin-Riou, we have now days precise conjectures on how the special values should behave $p$-adically; we fix a  regular sub-module of $V$. This corresponds to the choice of a sub-$(\varphi,N)$-module of $\mathcal{D}_{\st}(V)$ which is a section of the exponential map 
\begin{align*}
\mathcal{D}_{\st}(V) \rightarrow t(V) \cong \frac{\mathcal{D}_{\st}(V)}{\mathrm{Fil}^0\mathcal{D}_{\st}(V)}.
\end{align*}
Let $h$ be the valuation of the determinant of $\varphi$ on $D$. We can state the following conjecture
\begin{conj}
There exists a formal series $L^D_p(V,T) \in \C_p[[T]]$ who grows as $\log_p^h$ such that for all non-trivial, finite-order characters $\eps: 1+p\Z_p \rightarrow \mu_{p^{\infty}}$ we have 
\begin{align*}
L^D_p(V,\eps(1+p)-1)= C_{\eps}(D)\frac{L(M,0)}{\Omega(M)}.
\end{align*}
Moreover, for $\eps=\mathbf{1}$ we have 
\begin{align*}
L_p(V,0)= E(D)\frac{L(M,0)}{\Omega(M)},
\end{align*}
where $E(D)$ is an explicit product of Euler-type factors depending on $D$ and ${(\mathcal{D}_{\st}(V)/D)}^{N=0}$.
\end{conj} 
It may happen that one of the factor of $E(D)$ vanishes and then we say that trivial zeros appear. Since the seminal work of \cite{MTT}, people have been interested in describing the $p$-adic derivative of $L^D_p(V,{(1+p)}^s-1)$ when trivial zeros appear.\\
We suppose for simplicity that $L(M,0)$ is not vanishing. We have the following conjecture by Greenberg and Benois;
\begin{conj}
Let $t$ the number of vanishing factors of $E(D)$. Then 
\begin{itemize}
	\item $\mathrm{ord}_{s=0}L^D_p(V,{(1+p)}^s-1) = t$,
	\item $L_p(V,0)^* = \Ll(V,D)E^*(D)\frac{L(M,0)}{\Omega(M)}$.
\end{itemize}
Here $E^*(D)$ is the product of non-vanishing factors of $E(D)$ and $\Ll(V,D)$ is a number defined in purely Galois theoretical terms (see Section \ref{defLinv}). 
\end{conj}
The error factor $\Ll(V,D)$ is quite mysterious. It has been calculated in only few cases  for the symmetric square of a (Hilbert) modular form by Hida, Mok and Benois and for symmetric power of Hilbert modular forms by Hida and Jorza-Harron. Unless $V$ is an elliptic curve over $\Q$ with multiplicative reduction at $p$ we can not prove the non-vanishing of $\Ll(V,D)$.\\
The aim of this paper is to calculated it in some new cases; let $F$ be a totally real field where $p$ is unramified and $\pi$ be an automorphic representation of ${\mathrm{GSp}_{2g}}_{/F}$. We suppose that it has Iwahoric level at all $\pfrak \mid p$. We suppose moreover that $\pi_{\pfrak}$ is either Steinberg (see Definition \ref{defstb}) or spherical. We partition consequently the prime ideals of $F$ above $p$ in $S^{\mathrm{Stb}}\cup S^{\mathrm{Sph}}$. \\
We have conjecturally two Galois representations associated to $\pi$, namely the spinorial one $V_{\mathrm{spin}}$ and the standard one $V_{\mathrm{sta}}$. Let $V$ be one of these two representations. We choose for each prime $\pfrak$ of $F$ dividing $p$ a  regular sub module $D_{\pfrak}$ of $\mathcal{D}_{\st}(V_{\vert_{G_{F_{\pfrak}}}})$.\\
Consider a family of Siegel-Hilbert modular forms as in \cite{UrbEig} passing through $\pi$. Let us denote by $\beta_{\pfrak}(\kappa)$ the eigenvalue of the normalized Hecke operators $U_{1,\pfrak}$ (see Definition \ref{Upi}) on this family. Let $S^{\mathrm{Sph},1}=S^{\mathrm{Sph},1}(V,D)$ be the subset of $S^{\mathrm{Sph}}$ for which ${(\mathcal{D}_{\st}(V_{\pfrak})/D_{\pfrak})}^{N=0}$ does not contain the eigenvalue $1$. Conjecturally, it is empty for the spin representation. The eigenvalues $1$ always appears in $\mathcal{D}_{\st}(V_{\pfrak})$ for $V$ the standard representation but it may appear in $D_{\pfrak}$ (this is already the case for the symmetric square of a modular form). \\
Let $t_{\mathrm{Stb}}$ be the cardinality of $S^{\mathrm{Stb}}$ and  $t_{\mathrm{Sph}}$ be the cardinality of $S^{\mathrm{Sph},1}$. We define $f_{\pfrak}=[F^{\mathrm{ur}}_{\pfrak}:\Q_p]$.
\begin{theo}\label{teoStb}
Let $\pi$ be as above, of parallel weight $\underline{k}$. Let $V=V_{\mathrm{spin}}$ and suppose hypothesis {\bfseries LGp} of Section \ref{EffCal}, then the expected number of trivial zero for $L^D_p(V(k-1),T)$ is $t_{\mathrm{Stb}}$ and \begin{align*}
\Ll(V(k-1),D)= \prod_{\pfrak \in S^{\mathrm{Stb}}} -\frac{1}{f_{\pfrak}}{\frac{\textup{d}\log_p\beta_{\pfrak}(k)}{\textup{d}k}}_{\vert_{k=\underline{k}}}.
\end{align*}
Let $V=V_{\mathrm{std}}$, then the conjectural number of trivial zero for $L^D_p(V,T)$ is $t_{\mathrm{Stb}}+t_{\mathrm{Sph}}$ and \begin{align*}
\Ll(V,D)=\Ll(V,D)^{\mathrm{Sph}} \prod_{\pfrak \in S^{\mathrm{Stb}}} -\frac{1}{f_{\pfrak}}{\frac{\textup{d}\log_p \beta_{\pfrak}(k)}{\textup{d}k}}_{\vert_{k=\underline{k}}},
\end{align*}
where $\Ll(V,D)^{\mathrm{Sph}}$ is {\it a priori} global factor. It is $1$ if $t_{\mathrm{Sph}}=0$.
\end{theo}
In Section \ref{EffCal} we shall provide also  a formula for the $\Ll$-invariant of $V_{\mathrm{std}}(s)$ ($\mathrm{min}(k-g-1,g-1) \geq s \geq 1$).\\ 

The proof of the theorem is not different from the one of \cite[Theorem 2]{BenLinv2} which in turn is similar to the original one of \cite{SSS}.\\
Let now $g=2$.  Let $t$ be the number of primes above $p$ in $F$. We consider the $2t$-dimensional eigenvariety for ${\mathrm{GSp}_4}_{/F}$ with variables $k=\lgr k_{\pfrak,1},k_{\pfrak,2}\rgr_{\pfrak}$(see Section \ref{Ad}) and let us denote by $F_{\pfrak,i}(k)$ ($i=1,2$) the first two graded pieces of $\boldD^{\dagger}_{\mathrm{rig}}(V_{\spin})$. The $10$-dimensional Galois representation $\mathrm{Ad}(V_{\spin})$ has a natural regular sub-$(\varphi,N)$-module induced by the $p$-refinement of $\boldD^{\dagger}_{\mathrm{rig}}(V_{\spin})$ and which we shall denote by $D_{\Ad}$. With this choice of regular sub module, $\mathrm{Ad}(V_{\spin})$ presents $2t$ trivial zeros. In Section \ref{Ad} we prove the following theorem;
\begin{theo}\label{teoAd}
Let $\pi$ be an automorphic form of weight $\underline{k}$ and suppose hypothesis {\bfseries LGp} of Section \ref{EffCal} is verified for $V_{\mathrm{spin}}$, we have then 
\begin{align*}
\Ll(\mathrm{Ad}(V_{\spin}(\pi)),D_{\Ad}) =\prod_{\pfrak} \frac{2}{ f^2_{\pfrak}} {\mathrm{det}{\left(\begin{array}{cc}
\frac{\partial \log_p F_{\pfrak_i,1}(k)}{\partial k_{\pfrak_j,1}} & \frac{\partial \log_p F_{\pfrak_i,2}(k)}{\partial k_{\pfrak_j,1}} \\
\frac{\partial \log_p F_{\pfrak_i,1}(k)}{\partial k_{\pfrak_j,2}} & \frac{\partial \log_p F_{\pfrak_i,2}(k)}{\partial k_{\pfrak_j,2}}
\end{array}  \right)}_{1 \leq i,j \leq t}}_{\vert_{k=\underline{k}}}.
\end{align*}
\end{theo}
We remark that this theorem is the first to really go beyond $\mathrm{GL}_2$ and its representations $\mathrm{Sym}^n$.\\

The motivation for Theorem \ref{teoStb} lies in a generalization of \cite{RosPB} to Siegel forms. In {\it loc. cit.} we use Greenberg-Stevens method to prove a formula for the derivative of the symmetric square $p$-adic $L$-function and calculate the analytic $\Ll$-invariant and the same method of proof can be generalized to finite slope Siegel forms thanks to the overconvergent Ma\ss{}-Shimura operators and overconvergent projectors of Z.~Liu's thesis.\\ 
With some work, it could also be generalized to totally real field where $p$ where is inert, as already done for the symmetric square \cite{RosH}.\\ 

We hope to calculate the $\Ll$-invariant for $V_{\mathrm{std}}$ and $\Ad(V_{\spin})$ for more general forms in a future work.\\

In Section \ref{reminder} we recall the theory of $(\varphi,\G)$-module over a finite extension of $\Q_p$. It will be used in Section \ref{SecLinv} to generalize the definition of the $\Ll$-invariant \`a la Greenberg-Benois to Galois representations $V$ over general number field $F$ (note that we do not suppose $p$ split or unramified). This definition does not require one to pass through $\mathrm{Ind}_F^{\Q}(V)$ to calculate the $\Ll$-invariant which in turn simplifies explicit calculation. We shall check that this definition coincides with Benois' definition for $\mathrm{Ind}_F^{\Q}(V)$.\\ 
We prove the above-mentioned theorems in Section \ref{SieHib} and \ref{Ad}, inspired mainly by the methods of \cite{HTate2}.\\



\paragraph{Acknowledgement} This paper is part of the author's PhD thesis and we would like to thank J. Tilouine for his constant guidance and attention. We would like to thank A. Jorza for telling us that the study of the $\Ll$-invariant in the Steinberg case was within reach. We would also like to thank D.~Hansen, \'E.~Urban and S.~Shah for useful conversations. 


\section{\texorpdfstring{Some results on rank one $(\varphi,\G)$-module}{Some results on rank one (phi,Gamma)-module}}\label{reminder}
Let $L$ be a finite extension of $\Q_p$. The aim of this section is to recall certain results concerning $(\varphi,\G)$-modules over the Robba ring $\rr_L$. Let $L_0$ be the maximal unramified extension contained in $L$. Let $L_0'$ be the maximal unramified extension contained in $L_{\infty}:=L(\mu_{p^{\infty}})$ and $L'=L \cdot L_0'$.  Let $e_L:=[L(\mu_{p^{\infty}}):L_0(\mu_{p^{\infty}})]=[\G_{\Q_p}:\G_{L}]$, where $\G_{L}:=\mathrm{Gal}(L_{\infty}/L)$. We define 
\begin{align*}
\boldB_{L,\mathrm{rig}}^{\dagger,r}= & \lgr f=\sum_{n \in \Z} a_n \pi_L^n \vert a_n \in L_0', \mbox{ such that } f(X)=\sum_{n \in \Z} a_n X^n \right.\\
& \left. \;\:\: \mbox{is holomorphic on } p^{-\frac{1}{e_Lr}} \leq |X|_p <1 \rgr, \\
\rr_L = & \bigcup_{r}\boldB_{K,\mathrm{rig}}^{\dagger,r},
\end{align*}
where $\pi_L$ is a certain uniformizer coming from the theory of field of norm.
We have an action of $\varphi$ on $\rr_L$. If $L=L_0$, there is no ambiguity and we have:
\begin{align*}
\varphi(\pi_L)={(1+\pi_L)}^p-1, \;\;\;\; \varphi(a_n)=\varphi_{L_0'}(a_n).
\end{align*}
Otherwise the action on $\pi_L$ is more complicated.\\
Similarly, we have a $\G_L$-action. If $L=L_0$ we have
\begin{align*}
\gamma(\pi_L)={(1+\pi_L)}^{\cycar(\gamma)}-1,
\end{align*}
where $\cycar$ is the cyclotomic character. If $L$ is ramified we also have an action of $\G_L$ on the coefficients given by
\begin{align*}
\gamma(a_n)=\sigma_{\gamma}(a_n)
\end{align*}
where $\sigma_{\gamma}$ is the image of $\gamma$ via 
\begin{align*}
\G_L \rightarrow \G_L/\G_{L'} \stackrel{\cong}{\rightarrow} \mathrm{Gal}(L_0'/L_0).
\end{align*}
If $a_n$ is fixed by $\varphi$ and $\G_L$, then is it in $\Q_p$. We have  $\mathrm{rk}_{\rr_{\Q_p}}\rr_L = [L_{\infty}:\Q_{p,\infty}]$.\\
 Let $\delta:L^{\times} \rightarrow E^{\times}$ be a continuous character. We define $\rr_L (\delta)$ to be the rank one $(\varphi,\G_L)$-module with basis $e_{\delta}$ for which $\varphi(e_{\delta})= \delta(\pi_L)e_{\delta}$ and $\gamma (e_{\delta})=\delta(\cycar(\gamma))e_{\delta}$.\\ 

We classify now the cohomology of such a $(\varphi,\G_L)$-modules. It will be useful to calculate it explicitly in terms of $C_{\varphi,\gamma}$-complexes \cite[\S 1.1.5]{BenLinv}. We fix then a generator $\gamma_L$ of $\G_L$; if clear from the context, we shall drop the subscript $\phantom{e}_L$ and write simply $\gamma$.
\begin{prop}\label{dimH^i}
We have $H^{0}(\rr_L (\delta))=0$ unless $\delta(z)=\prod_{\tau}\tau(z)^{m_\tau}$ with $m_{\tau} \leq 0$ for all $\tau$; in this case we have $H^{0}(\rr_L \delta) \cong E$. We shall denote its basis by $t^m \otimes e_{\delta}$, where 
\begin{align*}
t^m =\oplus t^{m_{\tau}} \in \oplus_{\tau} B_{\mathrm{dR}}^+ \otimes_{L,\sigma} E.
\end{align*}
If $\delta(z)=\prod_{\tau}\tau(z)^{m_\tau}$ with $m_{\tau} \leq 0$, then $$\mathrm{dim}_E H^{1}(\rr_L (\delta)) = [L:\Q_p] +1.$$
If $\delta(z)=|\mathrm{N}_{L/\Q_p}(z)|_p\prod_{\tau}\tau(z)^{k_\tau}$ with $k_{\tau} \geq 1$, then $$\mathrm{dim}_E H^{1}(\rr_L (\delta)) = [L:\Q_p] +1.$$
Otherwise $$\mathrm{dim}_E H^{1}(\rr_L (\delta)) = [L:\Q_p].$$
We have $H^{2}(\rr_L (\delta))=0$ unless $\delta(z)=|\mathrm{N}_{L/\Q_p}(z)|_p\prod_{\tau}\tau(z)^{k_\tau}$ with $k_{\tau} \geq 1$; in this case we have $H^{2}(\rr_L (\delta)) \cong E$. 
\end{prop}
Note that when we choose $t^m$ as a basis we are implicitly using the fact that we can embed certain sub-rings of $\rr_L$ into $B_{\mathrm{dR}}^+$ (see \cite[\S 1.2.1]{BenLinv}).
\begin{proof}
The same results is stated in \cite[Proposition 2.14, 2.15]{NakPhi} for $E-B$-pairs, but the proof for $(\varphi,\G)$-modules is the same.\\  Recall that have a canonical duality \cite{LiuCD} given by cup product \begin{align*}
H^i(D)\times H^{2-i}(D^*(\cycar)) \rightarrow H^2(\cycar).
\end{align*}
The last fact is then a direct consequence. 
\end{proof}
This allows us to define a canonical basis of $H^2(\rr_L (|\mathrm{N}_{L/\Q_p}(z)|_p\prod_{\tau}\tau(z)^{k_\tau}))$. \\
We define $H^1_{\mathrm{f}}(D)$ has  the $H^1$ of the complex 
\begin{align*}
\mathcal{D}_{\cris}(D) \rightarrow t_D \oplus \mathcal{D}_{\cris}(D)
\end{align*}
and we have immediately \cite[Proposition 2.7]{NakPhi}
\begin{align}\label{dimH1f}
\mathrm{dim}_{E}H^1_{\mathrm{f}}(D) = \mathrm{dim}_{E}(H^0(D)) + \mathrm{dim}_{E}t_D.
\end{align}
Hence 
\begin{lemma}
If $\delta(z)=\prod_{\tau}\tau(z)^{m_\tau}$ with $m_{\tau} \leq 0$, then $$\mathrm{dim}_E H^{1}_{\mathrm{f}}(\rr_L (\delta)) = 1.$$
If $\delta(z)=|\mathrm{N}_{L/\Q_p}(z)|_p\prod_{\tau}\tau(z)^{k_\tau}$ with $k_{\tau} \geq 1$, then $$\mathrm{dim}_E H^{1}_{\mathrm{f}}(\rr_L (\delta)) = d.$$
\end{lemma}
We now want to calculate $H^{1}_{\mathrm{f}}(\rr_L (\delta))$ for $\delta(z)=\prod_{\tau}\tau(z)^{m_\tau}$ with $m_{\tau} \leq 0$. We recall the following lemma \cite[Lemma 1.4.3]{BenLinv}
\begin{lemma}
The extension in $H^{1}(\rr_L (\delta))$ corresponding to the couple $(a,b)$ is crystalline if and only if the equation $(1-\gamma)x=b$ has a solution in $D\left[\frac{1}{t}\right]$
\end{lemma}
The following proposition in an immediate consequence of the above lemma  \cite[Theorem 1.5.7 (i)]{BenLinv} (see also the construction of \cite{NakPhi} at page 900)
\begin{prop}\label{basef}
Let $e_{\delta}$ be a basis for $\rr_L (\delta)$. Then $x_m=\mathrm{cl}(t^m,0)e_{\delta}$ is a basis of $H^{1}_{\mathrm{f}}(\rr_L (\delta))$.
\end{prop}
\begin{rem}
If $\delta$ is the trivial character then $x_0$ corresponds (via class field theory) to the unramified $\Z_p$-extension of $ \mathrm{Hom}(G_L,E^{\times})\cong H^1(G_L,E) $.
\end{rem}
We have now to cut out a ``canonical'' one-dimensional subspace in $H^{1}(\rr_L (\delta))$ which trivially intersects $H^{1}_{\mathrm{f}}(\rr_L (\delta))$ (and reduces to the cyclotomic $\Z_p$-extension in the sense of the previous remark).\\
We introduce another extension. We define $y_k=\frac{1}{e_L}\log(\cycar(\gamma_L))\mathrm{cl}(0,t^m)e_{\delta}$. \\
We explain why this cocycle is of interest for us. 
We can calculate cohomology of induced $(\varphi,\G_{\Q_p})$-module. Indeed, we consider now two $p$-adic fields $K$ and $L$, $L$ a finite extension of $K$. The main reference for this part is \cite{LiuCD}. Let $D$ be a $(\varphi,\G_L)$-module, we define 

\begin{align*}
\mathrm{Ind}_{\G_L}^{\G_K}(D)=\lgr f: \G_K \rightarrow D \vert f(hg)=hf(g) \:\: \forall h \in \G_L \rgr.
\end{align*}
It has rank $[L:K]\mathrm{rk}_{\rr_L}(D)$ over $\rr_K$; indeed $\rr_L$ is a $\rr_K$-module of rank $[L:K]/|\G_K/\G_L|$. (The unramified part of $L/K$ plus the ramified part which is disjoint by $K_{\infty}$. See after \cite[Theorem 2.2]{LiuCD}.)
If $D$ comes from a $G_L$-representation $V$ we have 
\begin{align*}
\boldD^{\dagger}_{\mathrm{rig}}(\mathrm{Ind}_{G_L}^{G_K}(V)) = \mathrm{Ind}_{\G_L}^{\G_K}(\boldD^{\dagger}_{\mathrm{rig}}(V)).
\end{align*}
We have then the equivalent of Shapiro's lemma
\begin{align*}
H^i(D) \cong H^i(\mathrm{Ind}_{\G_L}^{\G_K}(D)).
\end{align*}
 Moreover, the aforementioned duality for $(\varphi,\G)$-modules is compatible with induction \cite[Theorem 2.2]{LiuCD}.\\
If $D\cong \rr_L(\delta)$ is free of rank one, then we have an explicit description of $\mathrm{Ind}_{\G_L}^{\G_K}(D)$. Let $e_{\infty}=|\G_K/\G_L|$, we write $\lgr \omega^{i} \rgr_{i=0}^{e_{\infty}-1}$ for $(\G_K/\G_L)^{\wedge}$. The $\mathrm{Ind}_{\G_L}^{\G_K}(D)$ is the $\rr_L$-span of $f_i$, where $f_i(g)=\omega^i(g)\delta(\cycar(g)) e_{\delta}$.\\ 

We go back to the previous setting, where $K=\Q_p$ (hence $e_{\infty}=e_L$). We have the following exact sequence 
\begin{align}\label{injinduction}
0 \rightarrow \rr_{\Q_p}(z^{\sum m_{\sigma}}) \rightarrow \mathrm{Ind}_{L}^{\Q_p}(\rr_L (\delta))\rightarrow D' \rightarrow 0,
\end{align}
where we send $e_{z^{\sum m_{\tau}}}$ to $e_{\delta}$ (the image is $\rr_{\Q_p} e_{\delta}$). Note that $z^{\sum m_{\tau}}$ is the restriction of $\delta$ to $\Q_p^*$. Note that $H^0(D')=0$ by calculation similar to \cite[Proposition 2.1]{ColTri} or \cite[Proposition 2.14]{NakPhi}. The long exact sequence in cohomology induces 
\begin{align*}
H^0(\rr_{\Q_p}(z^{\sum m_{\tau}}) ) & \stackrel{\cong}{\rightarrow} H^0(\mathrm{Ind}_{L}^{\Q_p}(\rr_L (\delta))),\\
H^1(\rr_{\Q_p}(z^{\sum m_{\tau}}) ) & \hookrightarrow H^1(\mathrm{Ind}_{L}^{\Q_p}(\rr_L (\delta))).
\end{align*} 
The first morphism is the identity: the fixed basis $t^{\sum m_{\tau}}e_{z^{\sum m_{\tau}}}$ is sent to the basis $t^k e_{\delta}$ (because $E\otimes_{\Q_p} L \cong \oplus_{\tau} E$).\\
The cocycle $x_m$ is the image of $x_{\sum  m_{\tau}}$. The image of  $y_{\sum m_{\tau}} \in H^1_{\mathrm{c}}(\rr_{\Q_p}(z^{\sum k_{\sigma}}))$ is instead $y_m$ (if $\gamma_L =\gamma_{\Q_p}^{e_L}$).

\begin{prop}
Let $D$ a $(\varphi,\G)$-module over $\rr_L$ with non-negative Hodge-Tate weight. Suppose that $\mathcal{D}_{\st}(D)=\mathcal{D}_{\st}(D)^{\varphi=1}$. Then $D$ is crystalline and
\begin{align*} D \cong \oplus \rr_L (\delta_i)
\end{align*}
with $\delta_i(z)=\prod_{\tau}\tau(z)^{k_{i,\tau}}$.
\end{prop}
\begin{proof}
We follow closely the proof \cite[Proposition 1.5.8]{BenLinv}. As $N\varphi=p\varphi N$ we obtain immediately that $N=0$, hence $D$ is crystalline.\\
Let $r$ be the rank of $D$ over $\rr_L$. We write the Hodge-Tate weight as $(k_i)_{i=1}^r$ where $k_i={(k_{i,\tau})}_{\tau}$ and $k_i \leq k_{i+1}$. We prove it by induction; the case $r=1$ is clear. \\
For $r=2$ we can suppose $k_1=0$ by twisting. Let $\delta$ be defined by $\prod_{\tau}\tau(z)^{k_{\tau}}$. So we have an extension of $\rr_L (\delta)$ by $\rr_L$.  Let $m_2$ be a lift to $D$ of a basis of $\rr_1$. As $\varphi=1$ we have $\varphi m_2=m_2$. As the extension is crystalline we know  that $\gamma$ acts trivially too, hence the extension splits.\\
Suppose now $r >2$. Take $v$ in $\mathrm{Fil}^{k_d}\mathcal{D}_{\st}$. Define $\delta_d$ by $\delta_d(z)=\prod_{\tau}\tau(z)^{k_{d,\tau}}$, we have 
\begin{align*}
0 \rightarrow \rr_L (\delta_d) \rightarrow D \rightarrow D' \rightarrow 0.
\end{align*}
By inductive hypothesis $D' \cong \oplus_{i=1}^{d-1} \rr_L (\delta_i)$. We can write 
\begin{align*}
\mathrm{Ext}(D', \rr_L (\delta_d)) = \oplus_{i=1}^{d-1}\mathrm{Ext}(\rr_L (\delta_i), \rr_L (\delta_d))
\end{align*}
and we are reduced to the case $r=2$ which has already been dealt.
\end{proof}

We consider a $(\varphi,\G)$-module $M$ which sits in the non-split exact sequence 
\begin{align}\label{caseM}
0 \rightarrow M_0:=\oplus_{i=1}^r \rr_L (\delta_i) \rightarrow M \rightarrow M_1:= \oplus_{i=1}^r \rr_L (\delta_i') \rightarrow 0, 
\end{align}
where $\delta_i(z)=|\mathrm{N}_{L/\Q_p}(z)|_p\prod_{\tau}\tau(z)^{m_{i,\tau}}$ with $m_{i,\tau} \geq 1$ for all $\tau$ and $\delta_i'(z)=\prod_{\tau}\tau(z)^{k_{i,\tau}}$ with  $k_{i,\tau} \leq 0$ for all $\tau$. We say that $M$ is of type $U_{m,k}$ if the image of $M$ in $H^1(M_1)$ is crystalline. 
\begin{prop}\label{Mstruct}
Suppose that $M$ is not of type $U_{m,k}$. Then we have $\mathrm{dim}_E(H^1(M))=2[L:\Q_p]r$ and $H^2(M)=H^0(M)=0$. Moreover, if we write
\begin{align*}
0 \rightarrow H^0(M_1) \stackrel{\Delta_0}{\rightarrow} H^1(M_0) \stackrel{f_1}{\rightarrow} H^1(M) \stackrel{g_1}{\rightarrow} H^{1}(M_1) \stackrel{\Delta_1}{\rightarrow} H^{2}(M_0)\rightarrow 0
\end{align*}
we have $H^{1}(M_0)=\mathrm{Im}(\Delta_1)\oplus H^{1}_{\mathrm{f}}(M_0)$, $\mathrm{Im}(f_1)=H^1_{\mathrm{f}}(M)$ and $H^1(M_1)=\mathrm{Im}(g_1)\oplus H^{1}_{\mathrm{f}}(M_1)$.
\end{prop}
\begin{proof}
We have $H^0(M)=0$. Note that $M^*(\cycar)$ is a module of the same type, hence $H^2(M)=H^0(M^*(\cycar))=0$. We can write 
\begin{align*}
0 \rightarrow H^0(M_1) \rightarrow H^1(M_0) \stackrel{f_1}{\rightarrow} H^1(M) \stackrel{g_1}{\rightarrow} H^{1}(M_1) \rightarrow H^{2}(M_0)\rightarrow 0
\end{align*}
and conclude by Proposition \ref{dimH^i}. \\ 
Note that $\mathrm{dim}_E H^{1}_{\mathrm{f}}(M)=rd$ by (\ref{dimH1f}).\\
By hypothesis, we have that $\mathrm{Im}(\Delta_1) \cap H^{1}_{\mathrm{f}}(M_0)=0 $ and the first statement follows from dimension counting.\\
The third statement follows from duality.\\
For the second statement $H^{1}_{\mathrm{f}}(M_0)$ injects into $H^{1}_{\mathrm{f}}(M)$. As both have the same dimension, we conclude.
\end{proof}
Suppose now that $M_0=\rr_L(\vert \mathrm{N}_{L/\Q_p}(z) \vert_{p}\prod_{\tau}z^{k_{\tau}})$ and $M_1=\rr_L(\prod_{\tau}z^{m_{\tau}})$, we give the following key proposition for the definition of the $\Ll$-invariant
\begin{lemma}\label{surjc}
The intersection of $T:=\mathrm{Im}(H^1(M))$ and $\mathrm{Im}(H^1(\rr_{\Q_p}(z^{\sum_{\tau}m_{\tau}})))$ in $\mathrm{Im}(H^1(M_1))$ is one dimensional.
\end{lemma}
\begin{proof}
The intersection is non-empty as the sum of their dimension is $d+2$ and $\mathrm{Im}(H^1(M_1))$ has dimension $d+1$. We have that $H^{1}_{\mathrm{f}}(M_1)$ is contained in $\mathrm{Im}(H^1(\rr_{\Q_p}(z^{\sum_{\tau}m_{\tau}})))$ and by the previous proposition the former is not in the image of $g_1$ and we are done.
\end{proof}
In particular, we deduce that $T$ surjects into $\mathrm{Im}(H_{\mathrm{c}}^1(\rr_{\Q_p}(z^{\sum_{\tau}m_{\tau}})))$.


\section{\texorpdfstring{$\Ll$-invariant over number fields}{Definition of the L-invariant over number fields}}\label{SecLinv}
Let $F$ be a number field. We consider a global Galois representation 
\begin{align*}
V : G_F \rightarrow \mathrm{GL}_n(E)
\end{align*}
where $E$ is $p$-adic field. We suppose that it is unramified outside a finite number of places $S$ containing all the $p$-adic places. We suppose moreover that it is semistable at all places above $p$ ({\it i.e.} $\mathcal{D}_{\st}(V_{\vert_{F_\pfrak}})$ is of rank $n$ over $F^{\mathrm{ur}}_{\pfrak} \otimes_{\Q_p} E$, being $F^{\mathrm{ur}}_{\pfrak}$ the maximal unramified extension of $\Q_p$ contained in $F^{\mathrm{ur}}_{\pfrak}$).\\
In this section we generalize Greenberg-Benois definition of the $\Ll$-invariant to such a $V$ when it presents trivial zeros. Note that we do not require $p$ split or unramified in $F$.\\
Let $t$ be the number of trivial zeros. The classical definition by Greenberg \cite{TTT} defines the $\Ll$-invariant as the ``slope'' of a certain $t$-dimension subspace of $H^1(G_{\Q_p},\Q_p^t)$ which is a $2t$-dimensional space with a canonical basis given by $\mathrm{ord}_p$ and $\log_p$.\\
In our setting, the main obstacle is that the cohomology of the trivial $(\varphi,\G)$-module $\rr_{F_{\pfrak}}$ is no longer two-dimensional and it is not immediate to find a suitable subspace. Inspired by Hida's work for symmetric powers of Hilbert forms \cite{HTate2}, we consider the image of $H^1(\rr_{\Q_p})$ inside $H^1(\rr_{F_{\pfrak}})$. \\
If $t$ denotes the number of expected trivial zeros, we show that we can define, similarly to \cite{BenLinv}, a $t$-dimensional subspace of $H^1(G_{\Q,S},V)$ whose image in $H^1(\rr_{\Q_p})$ has trivial intersection with the crystalline cocycle. This is enough to define the $\Ll$-invariant; we further check that our definition is compatible with Benois'.

\subsection{\texorpdfstring{Definition of the $\Ll$-invariant}{Definition of the L-invariant}}\label{defLinv}
We define local cohomological conditions $L_v$ in order to define a Selmer group; we denote by $G_v$ a fixed decomposition group at $v$ in $G_{F,S}$ and by $I_v$ the inertia. For $v \nmid p$ we define
 \begin{align*}
L_v:=\mathrm{Ker}\left(H^1(G_{v},V)\rightarrow H^1(I_v,V) \right).
\end{align*}
If $v \mid p$ we define
\begin{align*}
L_v:=H^1_{\mathrm{f}}(F_v,V) = \mathrm{Ker}(H^1(G_v,V)\rightarrow H^{1}(G_v, V \otimes_{E} \boldB_{\mathrm{cris}})).
\end{align*}
If $\boldD^{\dagger}_{\mathrm{rig}}(V)$ denotes the $(\varphi,\G)$-module associated with $V$ we also have $L_p = H_{\mathrm{f}}^{1}(\boldD^{\dagger}_{\mathrm{rig}}(V))$.
We define then the Bloch-Kato Selmer group 
\begin{align*}
H^1_{\mathrm{f}}(V) :=\mathrm{Ker}\left( H^1(G_{F,S},V) \rightarrow \prod_{v\in S}\frac{H^1(D_v,V)}{L_v} \right).
\end{align*}
We make the following additional hypotheses 
\begin{itemize}
\item[\bfseries{C1})] $H^1_{\mathrm{f}}(V)=H^1_{\mathrm{f}}(V^*(1))=0$,
\item[\bfseries{C2})] $H^0(G_{F,S},V)=H^0(G_{F,S},V^*(1))=0$,
\item[\bfseries{C3})] $\varphi$ on $\boldD_{\mathrm{st}}(V_{\vert_{F_{\pfrak}}})$ is semisimple at $1 \in F^{\mathrm{ur}}_{\pfrak} \otimes_{\Q_p} E$ and $p^{-1} \in F^{\mathrm{ur}}_{\pfrak} \otimes_{\Q_p} E$ for all $\pfrak \mid p$,
\item[\bfseries{C4})] $\boldD^{\dagger}_{\mathrm{rig}}(V_{\vert_{F_{\pfrak}}})$ has no saturated sub-quotient of type $U_{m,k}$ for all $\pfrak \mid p$.
\end{itemize}
Note that if $V$ satisfies the previous four conditions, so does $V^*(1)$.\\
The first two conditions tell us that the Poitou-Tate sequence reduces to 
\begin{align}\label{PT}
H^1(G_{F,S},V) \cong \bigoplus_{v \in S} \frac{H^1(D_v,V)}{H^1_{\mathrm{f}}(V,F_v)}.
\end{align}

For each $\pfrak \mid p$ we denote by $V_{\pfrak}$ the restriction to $G_{F_{\pfrak}}$ of $V$. We choose a regular sub-module $D_{\pfrak} \subset \boldD_{\mathrm{st}}(V_{\pfrak})$ and define a filtration $(D_{\pfrak,i})$ of $\boldD_{\mathrm{st}}(V_{\pfrak})$.
\begin{align}\label{filt} 
D_{\pfrak,i} = \left\{ \begin{array}{cc}
 0 & i=-2, \\
 (1-p^{-1}\varphi)D_{\pfrak} + N(D_{\pfrak}^{\varphi=1}) & i=-1,\\
 D_{\pfrak} & i=0,\\
D_{\pfrak} + {\boldD_{\mathrm{st}}(V_{\pfrak})}^{\varphi=1}\cap N^{-1}(D_{\pfrak}^{\varphi=p^{-1}}) & i=1,\\
 \boldD_{\mathrm{st}}(V_{\pfrak}) & i=2.
\end{array}
\right.
\end{align}
We have that $D_{\pfrak,1}/D_{\pfrak,-1}$ coincides with the eigenvectors of $\varphi$ on $\boldD_{\mathrm{st}}(V_{\pfrak})$ of eigenvalue $1$ resp.  $p^{-1}$ and which are in the kernel of $N$ resp. in the image of  $N$. \\
This filtration induces a filtration on $\boldD^{\dagger}_{\mathrm{rig}}(V_{\pfrak})$. Namely, we pose 
\begin{align*}
F_{i} \boldD^{\dagger}_{\mathrm{rig}}(V_{\pfrak}) = \boldD^{\dagger}_{\mathrm{rig}}(V_{\pfrak}) \cap (D_{\pfrak,i} \otimes \rr_{F_{\pfrak,\log}}[t^{-1}]).
\end{align*}
We define 
\begin{align*}
W_{\pfrak} := F_1 \boldD^{\dagger}_{\mathrm{rig}}(V_{\pfrak})/ F_{-1} \boldD^{\dagger}_{\mathrm{rig}}(V_{\pfrak}).
\end{align*}
We have
\begin{align*}
W_{\pfrak} = W_{\pfrak,0}\bigoplus W_{\pfrak,1} \bigoplus M_{\pfrak}
\end{align*}
where $t_{\pfrak,0}=\mathrm{dim}_{E} H^{0}(W)= \mathrm{rank}_{\rr_{F_{\pfrak}}} W_0 $, $t_{\pfrak,1}=\mathrm{dim}_{E} H^{0}(W^*(1))= \mathrm{rank}_{\rr_{F_{\pfrak}}} W_1$ and $M$ sits in a non split sequence 
\begin{align*}
0 \rightarrow M_{\pfrak,0} \stackrel{f}{\rightarrow} M_{\pfrak} \stackrel{g}{\rightarrow} M_{\pfrak,1} \rightarrow 0
\end{align*}
such that $\mathrm{gr}^0 (\boldD^{\dagger}_{\mathrm{rig}}(V_{\pfrak})) = W_{\pfrak,0} \oplus M_{\pfrak,0}$ and $\mathrm{gr}^1 (\boldD^{\dagger}_{\mathrm{rig}}(V_{\pfrak})) = W_{\pfrak,1} \oplus M_{\pfrak,1}$.\\
We can prove exactly in the same way as \cite[Proposition 2.1.7 (i)]{BenLinv} that {\bfseries{C4}} implies  $\mathrm{rank}_{\rr_{F_{\pfrak}}} M_1=\mathrm{rank}_{\rr_{F_{\pfrak}}} M_0$.\\
In order to define the $\Ll$-invariant  we shall follow verbatim Benois' construction.
For sake of notation, we write $\boldD^{\dagger}_{\pfrak}$ for $\boldD^{\dagger}_{\mathrm{rig}}(V_{\pfrak})$. We obtain from \cite[Proposition 1.4.4 (i)]{BenLinv} 
\begin{align*}
H^{1}_{\mathrm{f}}(\mathrm{gr}^2(\boldD^{\dagger}_{\pfrak}))= H^0(\mathrm{gr}^2(\boldD^{\dagger}_{\pfrak}))=0.
\end{align*}
We deduce the following isomorphism
\begin{align}\label{h1fMV}
H^{1}_{\mathrm{f}}(F_1 \boldD_{\pfrak}^{\dagger}) = H^{1}_{\mathrm{f}}(\boldD^{\dagger}_{\pfrak}) = H^{1}_{\mathrm{f}}(F_{\pfrak},V).
\end{align}
As the Hodge-Tate weights of $F_{-1}\boldD_{\pfrak}^{\dagger}$ are $<0$, we obtain from \cite[Proposition 1.5.3 (i)]{BenLinv} and Poiteau-Tate duality $H^2(F_{-1}\boldD_{\pfrak}^{\dagger})=0$. 
Using the long exact sequence associated to 
\begin{align*}
0\rightarrow F_{-1}\boldD_{\pfrak}^{\dagger} \rightarrow F_{1}\boldD_{\pfrak}^{\dagger} \rightarrow W_{\pfrak} \rightarrow 0
\end{align*}
we see that 
\begin{align*}
\frac{H^{1}(W_{\pfrak})}{H^1_{\mathrm{f}}(W_{\pfrak})}=  \frac{H^1(F_{-1}\boldD_{\pfrak}^{\dagger})}{H^{1}_{\mathrm{f}}(F_{\pfrak},V)}.
\end{align*}
We suppose now 
\begin{itemize}
\item[\bfseries{C5})] $W_{\pfrak,0}=0$ for all $\pfrak \mid p$.
\end{itemize}
Write $\mathrm{gr}^1 (\boldD^{\dagger}_{\pfrak})=\oplus_{i=1}^{t_{\pfrak,1} + r_{\pfrak}} \rr_{F_{\pfrak}}(\prod_{\tau_{\pfrak}} \tau_{\pfrak}(z)^{m_{i,\tau_{\pfrak}}})$. 
We define the $2( t_{\pfrak,1} + r_{\pfrak})$-dimensional subspace  obtained as the image of \begin{align}\label{defindp}
\mathrm{Ind}_{\pfrak} := \mathrm{Im}\left(H^{1}\left(\oplus_{i=1}^{t_{\pfrak,1} + r_{\pfrak}} \rr_{\Q_p}\left(z^{\sum_{\tau_{\pfrak}} m_{i,\tau_{\pfrak}}}\right)\right)\right) \subset H^{1}(\mathrm{gr}^1 (\boldD^{\dagger}_{\pfrak})).
\end{align}
We define 
\begin{align*}
T_{\pfrak} = (H^1(F_{1}\boldD_{\pfrak}^{\dagger}) \cap \mathrm{Ind}_{\pfrak})/H^{1}_{\mathrm{f}}(F_{\pfrak},V) . 
\end{align*}
It has dimension $t_{\pfrak,1} + r_{\pfrak}$.\\
Write $t=\sum_{\pfrak} t_{\pfrak,1} + r_{\pfrak}$. We have a unique $t$-dimensional subspace $H^1(D,V)$ of $H^1(G_{F,S},V)$ projecting via \ref{PT} to $\oplus_{\pfrak} T_{\pfrak}$. 
We have an isomorphism \cite[Proposition 1.5.9]{BenLinv}
\begin{align*}
H^1(\oplus_{i=1}^{t_{\pfrak,1} + r_{\pfrak}} \rr_{\Q_p}(z^{\sum_{\tau_{\pfrak}} m_{i,\tau_{\pfrak}}})) = \mathcal{D}_{\mathrm{cris}}(W_1\oplus M_1) \oplus \mathcal{D}_{\mathrm{cris}}(W_1\oplus M_1)
\end{align*}
We shall denote the two projections by $\iota_{\mathrm{f}}$ and $\iota_{\mathrm{c}}$.\\
 A canonical basis is given by the above mentioned cocycles $x_m$ (resp. $y_m$) defined in (resp. right after) Proposition \ref{basef}.\\
By abuse of notation, we still denote by $\iota_{\mathrm{f}}$ resp. $\iota_{\mathrm{c}}$ be the projection of $H^{1}(D,V)$ to $\mathcal{D}_{\mathrm{cris}}(W)$ via $\iota_{\mathrm{f}}$ resp. $\iota_{\mathrm{c}}$. By the remark after Lemma \ref{surjc} and the definition of $T_{\pfrak}$, we have that $H^{1}(D,V)$  surjects into  $\mathcal{D}_{\mathrm{cris}}$ via $\rho_{\mathrm{c}}$. \\
Summing up, we can give the following definition;
\begin{defin}\label{Linvar}
The $\Ll$-invariant of the pair $(V,D)$ is  
\begin{align*}
\Ll(D,V):= \mathrm{det}(\iota_{\mathrm{f}} \circ \iota^{-1}_{\mathrm{c}}),
\end{align*}
where the determinant is calculated w.r.t. the basis ${(x_{m_i},y_{m_j})}_{1\leq i,j \leq t }$.
\end{defin}
\begin{rem}
There is no {\it a priori} reason for which $\Ll(D,V)$ should be non-zero.
\end{rem}
In the case $W_{\pfrak}=M_{\pfrak}$ we see from the description of $H^1(F_{1}\boldD_{\pfrak}^{\dagger})$ that the space $T_{\pfrak}$ depends only on $V_{\vert_{F_{\pfrak}}}$ exactly as in the classical case.


\subsection{Comparison with Benois' definition}

Fix a global field $F$ and let $\lgr \pfrak \rgr$ be the set of primes above $p$. \\
Let $G_p$ denote a fixed decomposition group at $p$ in $G_{\Q}$ and let $\pfrak_0$ be the corresponding place of $F$. Let $G_{\pfrak_0,F}$ be the decomposition group at $\pfrak_0$ in $G_F$. For each other place $\pfrak$ above $p$ in $F$, we have $G_{\pfrak}=\sigma_{\pfrak}G_{p} \sigma_{\pfrak}^{-1}$. We shall denote by $G_{\pfrak,F}$ the corresponding decomposition group in $G_F$. Consider a $p$-adic Galois representation 
\begin{align*}
V : G_F \rightarrow \mathrm{GL}_n(E).
\end{align*}
We shall suppose $E$ big enough to contain the Galois closure of $F_{\pfrak}$, for all $\pfrak$. As before, we suppose $V$ semistable at all primes above $p$. We have then 
\begin{align*}
\mathrm{Ind}_F^{\Q}(V) \cong_{G_p} \bigoplus_{\pfrak} \sigma_{\pfrak}^{-1} \mathrm{Ind}_{G_{\pfrak,F}}^{G_{\pfrak}} V_{\vert_{G_{\pfrak,F}}}
\end{align*}
where $\sigma_{\pfrak} \in G_p \setminus \mathrm{Hom}(F,\Qb)$.\\
Consider the $(\varphi,\G)$-module
\begin{align*} 
\boldD^{\dagger}:=\boldD^{\dagger}_{\mathrm{rig}}\left(\mathrm{Ind}_F^{\Q} V\right).
\end{align*}
We let $D$ be the regular $(\varphi,N)$-module of $\mathcal{D}_{\st}(\boldD^{\dagger}) $ induced by $\lgr D_{\pfrak}\rgr_{\pfrak}$.
As before we have a filtration $(F_i \boldD^{\dagger})$ on $\boldD^{\dagger}$ induced by the filtration on $D$. We denote by $W$ the quotient $F_1 \boldD^{\dagger}/ F_{-1} \boldD^{\dagger}$. Note that it is semistable. We write $W=W_0\oplus M\oplus W_1$. We suppose {\bfseries{C1-C5}} of the previous section. 
\begin{lemma}
Let $M$ be as in (\ref{caseM}). We have 
\begin{align*}
0 \rightarrow \mathrm{Ind}(M_0) \rightarrow \mathrm{Ind}(M) \rightarrow \mathrm{Ind}(M_1)\rightarrow 0.
\end{align*}
\end{lemma}

We can now compare our definition of $\Ll$-invariant with Benois'.
\begin{prop}
We have a commutative diagram
$$
\xymatrix{
H^{1}(G_{\Q,S},\mathrm{Ind}(V))\ar[d] & H^1(\mathrm{Ind}(D),\mathrm{Ind}(V)) \ar[l]\ar[r]^-{\mathrm{Res}_p} \ar[d]  & \frac{H^{1}(F_1\boldD^{\dagger}(\mathrm{Ind}(V)))}{H^1_{\mathrm{f}}(G_p,\mathrm{Ind}(V))}=  \frac{H^1(F_{-1}\boldD^{\dagger})}{H^{1}_{\mathrm{f}}(G_{p},\mathrm{Ind}(V))}.
\ar[d]^{\iota_p} \\
H^{1}(G_{F,S},V)      & H^{1}(D,V)\ar[l]\ar[r]^-{\oplus_{\pfrak}\mathrm{Res}_{\pfrak}}       & \prod_{\pfrak} T_{\pfrak} }
$$
whose vertical arrows are isomorphism.
\end{prop}
\begin{proof}
We follow \cite[\S 3.4.4]{HIwa}. Recall that we wrote  $\boldD^{\dagger}_{\pfrak}$ for $\boldD^{\dagger}_{\mathrm{rig}}(V_{\pfrak})$. Shapiro's lemma tells us that 
\begin{align*}
\frac{H^1(G_p,\mathrm{Ind}_F^{\Q}V)}{H^1_{\mathrm{f}}(G_p,\mathrm{Ind}_F^{\Q}V)} \stackrel{\iota_p}{\cong} \bigoplus_{\pfrak} \frac{H^1(\boldD^{\dagger}_{\pfrak})}{H^1_{\mathrm{f}}(\boldD^{\dagger}_{\pfrak})}.
\end{align*}
We are left to show that $H^{1}(F_1\boldD^{\dagger}(\mathrm{Ind}(V)))$ is sent by $\mathrm{Res}_{\pfrak}$ into $(H^1(F_{1}\boldD_{\pfrak}^{\dagger}) \cap \mathrm{Inv}_{\pfrak})$ and we shall conclude by dimension counting.\\
We have then an injection 
\begin{align*}
F_1\boldD^{\dagger}(\mathrm{Ind}(V)) \hookrightarrow \oplus_{\pfrak} \mathrm{Ind}(F_1 (\boldD^{\dagger}_{\mathrm{rig}}(V_{\pfrak}))).
\end{align*}
Then clearly the image of $\iota_p$ lands in $H^1(F_{1}\boldD_{\pfrak}^{\dagger})$. But we have also the injection
\begin{align*}
\mathrm{gr}^1 (\boldD^{\dagger}_{\mathrm{rig}}(\mathrm{Ind}V)) \hookrightarrow \oplus_{\pfrak} \mathrm{Ind}(\mathrm{gr}^1 (\boldD^{\dagger}_{\mathrm{rig}}(V_{\pfrak})))
\end{align*}
induced by (\ref{injinduction}). Then the image of $\iota_p$ lands in $\mathrm{Inv}_{\pfrak}$ and we are done.
\end{proof}
\begin{prop}
We have $\Ll(V)=\Ll(\mathrm{Ind}_F^{\Q}(V))$.
\end{prop}
\begin{proof}
After the previous proposition, what we have to do is to notice  that the cocycle $x_{\sum m_{\tau}}$ (resp. $y_{\sum m_{\tau}}$) is identified with $x_{m}$ (resp. $y_m$).
\end{proof}

\section{Siegel-Hilbert modular forms, the local case}\label{SieHib}
The calculation of the $\Ll$-invariant requires to produce explicit cocycles in $H^1(D,V)$; when $V$ appears in $\mathrm{Ad}(V')$ for a certain representation $V'$ we can sometimes use the method of Mazur and Tilouine \cite{MT} to produce these cocycles. This has been done in many case for the symmetric square \cite{HLinv,MokLinv} and generalized to symmetric powers of the Galois representation associated with Hilbert modular forms in \cite{HTate2,HarJo}. The main limit of this approach is that for most of the representations $V$ is this computationally heavy to make it appear as the quotient of an adjoint representations.\\
In the case $\boldD^{\dagger}_{\mathrm{rig}}(V)=W=M$ the situation is way simpler; if $t=1$ it has been proved in \cite{BenLinv2} that to produce the cocycle in $H^1(V,D)$ it is enough to find deformations of $V\vert_{\Q_p}$.\\
We shall generalized the construction of Benois to our situation in the case $W_{\pfrak}=M_{\pfrak}$ and $r_{\pfrak}=1$. This will allow us to give a complete formula for the $\Ll$-invariant of the Galois representations associated with a Siegel-Hilbert modular form which is Steinberg at all primes above $p$. 

\subsection{\texorpdfstring{The case $t_{\pfrak}=r_{\pfrak}=1$}{The case t=r=1}}
We suppose now that $W_{\pfrak}=M_{\pfrak}$ and $r_{\pfrak}=1$. For sake of notation, in this section we shall drop the index $\phantom{er}_{\pfrak}$.\\
All that we have to do is to check that the calculation of \cite[Theorem 2]{BenLinv} works in our setting.\\
We write as before 
\begin{align*}
0 \rightarrow M_0 \rightarrow M \rightarrow M_1 \rightarrow 0
\end{align*}
and, only in this subsection, we shall write $\delta$ for the character defining $M_0$ and $\psi$ for the character defining $M_1$. We have $\delta(z)=|\mathrm{N}_{L/\Q_p}(z)|_p\prod_{\tau}\tau(z)^{k_\tau}$ with $k_{\tau} \geq 1$ and $\psi(z)=\prod_{\tau}\tau(z)^{m_\tau}$ with $m_{\tau} \leq 0$.
We consider an infinitesimal deformation 
\begin{align*}
0 \rightarrow M_{0,A} \rightarrow M_A \rightarrow M_{1,A} \rightarrow 0,
\end{align*} over $A=E[T]/(T^2)$.  We suppose that $M_{0,A}$ (resp. $M_{1,A}$) is an infinitesimal deformation of $M_0$ (resp. $M_1$).\\ 
We shall write $\delta_A$ and $\psi_A$ for the corresponding one-dimensional character.\\
\begin{theo}\label{LinvM}
Suppose that ${\textup{d} \log(\delta_A\psi_A^{-1})(\cycar(\gamma_{\Q_p})))}_{\vert_{T=0}}  \neq 0$; then 
\begin{align*}
\Ll(M,M_0) = - {\log(\cycar(\gamma_{\Q_p}))}\frac{{\textup{d} \log(\delta_A\psi_A^{-1})(p)}_{\vert_{T=0}}}{{\textup{d} \log(\delta_A\psi_A^{-1})(\cycar(\gamma_{\Q_p}))}_{\vert_{T=0}}}
\end{align*}
\end{theo}
\begin{proof}
Recall the definition of Ind in (\ref{defindp}). We have a vector $v=a x_m +by_m$ in $H^1(F_{1}\boldD^{\dagger}) \cap \mathrm{Ind}$. By definition $\Ll(M)=ab^{-1}$. 
The extension $M_{j,A}$ provides us with connecting morphisms $B_j^i:H^i(M_j)\rightarrow H^{i+1}(M_j)$.
We have by definition 
 \begin{align}
B_1^0(t^{-m}e_m)=&\mathrm{cl}(\textup{d}\mathrm{log}(\delta_A)(\pi_L)t^{-m}e_m,\textup{d}\mathrm{log}(\delta_A)(\cycar(\gamma))t^{-m}e_m)\label{B_1^0a}\\ 
 = &\textup{d}\mathrm{log}(\delta_A)(\pi_L)x_m + \textup{d}\mathrm{log}(\delta_A)(\cycar(\gamma))y_m. \label{B_1^0b}
\end{align}
As in \cite[\S 3.2]{BenLinv2} we consider the dual extension 
\begin{align*}
0 \rightarrow M_1^*(\cycar) \rightarrow M^*(\cycar) \rightarrow M_1^*(\cycar) \rightarrow 0,
\end{align*}
and we shall denote with a $\phantom{e}^*$ the corresponding map in the long exact sequence of cohomology.\\
 We have hence $\mathrm{ker}(\Delta_1) \bot \mathrm{Im}(\Delta_0^*)$ under duality, and a map 
 \begin{align*}H^1(M_1^*) \rightarrow H^1(\rr_{\Q_p}(|z|z^{1-\sum_{\tau}m_{\tau}})).\end{align*}
 By duality again, we deduce that the image of $\Delta_0^*$ inside the target of the above arrow is \begin{align*}
 a \alpha_{1-\sum_{\tau}m_{\tau}} + b \beta_{1-\sum_{\tau}m_{\tau}},\end{align*} 
 where $\alpha_{1-\sum_{\tau}m_{\tau}} $ (resp. $\beta_{1-\sum_{\tau}m_{\tau}}$) is the dual of $x_{\sum_{\tau}m_{\tau}}$ (resp. $y_{\sum_{\tau}m_{\tau}}$). \\
 We consider now the map  \begin{align*}{B_1^1}^*:H^1(M_1^*(\cycar))\rightarrow H^{2}(M_1^*(\cycar)) = H^2(\rr_{\Q_p}(|z|z^{1-\sum_{\tau}m_{\tau}})) \cong E,
 \end{align*}
 where the identity is the dual of the identity induced by (\ref{injinduction}). \\
 We can use \cite[Proposition 2.4]{BenLinv2} to see that after the above identification of $H^2$ with $E$ we have
 \begin{align}
 {B_1^1}^* (\alpha_{1-\sum_{\tau}m_{\tau}})& = c {\log_p(\cycar(\gamma_{\Q_p}))}^{-1} {\textup{d} \log_p(\delta_A)(\cycar(\gamma_{\Q_p}))}_{\vert_{T=0}}, \label{B_1^1a}\\
 {B_1^1}^* (\beta_{1-\sum_{\tau}m_{\tau}}) & = c  {\textup{d} \log_p(\delta_A)(p)}_{\vert_{T=0}}, \label{B_1^1b}
 \end{align}
where $c\in E^{\times}$.
We consider the following anti-commutative diagram 
$$
\xymatrix{
H^{0}(M_0^*(\cycar))\ar[d]^{{B_0^1}^*} \ar[r]^{\Delta_0^*} & H^1(M_1^*(\cycar)) \ar[d]^{{B_1^1}^*}  &  \\
H^{1}(M_0^*(\cycar))  \ar[r]^{\Delta_1^*} &  H^2(M_1^*(\cycar))}
$$
which implies 
\begin{align*}
{B_1^1}^* \Delta_0^* = - \Delta_1^*{B_0^1}^*.
\end{align*}
We calculate this identity on $t^{1-k}$. Applying (\ref{B_1^0a}) and (\ref{B_1^0b}) to ${\psi_A^{-1}\cycar}_{\vert_{\Q_p^{\times}}} $, (\ref{B_1^0a}) and (\ref{B_1^0b}) to ${\delta_A^{-1}\cycar}_{\vert_{\Q_p^{\times}}}$ and using \cite[(3.6)]{BenLinv2} which says 
\begin{align*}
\Delta_1^*{B_0^1}^*(t^{1-k}) =c\left(b\log_p(\delta_A)(p)+ a \textup{d}\log_p(\delta_A)(\cycar(\gamma)) \right)
\end{align*} we get 
\begin{align*}
b^{-1}a  = & - { \log_p(\cycar(\gamma_{\Q_p}))}\frac{{\textup{d} \log_p(\delta_A\psi_A^{-1})(p)}_{\vert_{T=0}}}{{\textup{d} \log_(\delta_A\psi_A^{-1})(\cycar(\gamma_{\Q_p}))}_{\vert_{T=0}}}.
\end{align*}
\end{proof}
\begin{rem}
In particular, this theorem proves that this definition of $\Ll$-invariant is compatible with the Coleman or Fontaine-Mazur ones \cite{PottLinv,YCZhang}.
\end{rem}

\subsection{\texorpdfstring{Calculation of the $\Ll$-invariant for Steinberg forms}{Calculation of the L-invariant for Steinberg forms}}\label{EffCal}
 We fix a totally real field $F$. Let $I$ be the set of real embeddings. Fix two embeddings
 \begin{align*}
 \C_p  \hookleftarrow \Qb \hookrightarrow\C
 \end{align*}
 as before. We partition $I = \sqcup_{\pfrak} I_{\pfrak}$ according to the $p$-adic place which each embedding induces. We shall denote by $q_{\pfrak}$ the residual cardinality for each prime ideal $\pfrak$. We consider an irreducible representation $\pi$ of ${\mathrm{GSp}_{2g}}_{/F}$ algebraic of weight $k=(k_{\tau})_{\tau}$, where $(k_{\tau})=(k_{\tau,1},\ldots,k_{\tau,g};k_{0})$ ($k_0$ is a parallel weight for $\mathrm{Res}_F^{\Q}(\mathbb{G}_m)$). With $k_{\tau,1}\leq k_{\tau,2} \ldots \leq k_{\tau,g}$. If $k_{\tau,1}\geq g +1$, then the weight is cohomological. The cohomological weight of $\pi$ is then \begin{align*}
{(\mu_{\tau})}_{\tau}=(k_{\tau})_{\tau}-{(g+1,\ldots,g+1;0)}_{\tau}.
\end{align*}
For parallel weight $k$, we shall choose $k_{0} = gk$.\\
 We describe now the conjectural Galois representation associated with $\pi$. We have a spin Galois representation $V_{\mathrm{spin}}$ (whose image is contained in $\mathrm{GL}_{2^g}$) and a standard Galois representation $V_{\mathrm{sta}}$ (whose image is contained in $\mathrm{GL}_{2g+1}$) given respectively by the spinoral and the standard representation of $\mathrm{GSpin}_{2g+1}= \phantom{e}^{L}\mathrm{GSp}_{2g}$. \\ 
  Thanks to the work of Scholze \cite{Scho} we dispose now of the standard Galois representation (see for example \cite[Theorem 18]{HarJo}). We also know the existence of the spin representation in many cases \cite{KretShin}.\\ 
We recall now some expected properties of these Galois representations. Our main reference is \cite[\S 3.3]{HarJo}. We will make the following assumption on $\pi$ at $p$; 
\begin{center}
for each $\pfrak \mid p$ either $\pi_{\pfrak}$ is spherical or Steinberg.
\end{center} 
We explain what we mean by Steinberg. Consider the Satake parameters at $\pfrak$, normalized as in \cite[Corollary 3.2]{BS}, $(\alpha_{\pfrak,1},\ldots,\alpha_{\pfrak,g})$. 
We have the following theorem on Iwahori spherical representation of $\mathrm{GSp}_{2g}(F_{\pfrak})$ \cite[Theorem 7.9]{Tadic}
\begin{theo}
Let $\alpha_1,\ldots, \alpha_g, \alpha$ be $g +1$ character of $F_{\pfrak}^{\times}$. Let $B_{\mathrm{GSp}_{2g}}$ be the Borel subgroup of $\mathrm{Sp}_{2g}(F_{\pfrak})$. Then $\mathrm{Ind}^{\mathrm{GSp}_{2g}(F_{\pfrak})}_{B_{\mathrm{GSp}_{2g}}}(\alpha_1 \times \cdots \times \alpha_g \rtimes \alpha)$ is not irreducible if and only if one of the following conditions is satisfied:
\begin{itemize}
\item[i)] There exist at least three indexes $i$ such that $\alpha_i$ has exact order two and the $\alpha_i$'s are mutually distinct;
\item[ii)] There exists $i$ such that $\alpha_i = {\vert \NN(\phantom{c}) \vert_{\pfrak}}^{\pm 1}$;
\item[iii)] There exist $i$ and $j$ such that $\alpha_i={\vert \NN(\phantom{c}) \vert_{\pfrak}}^{\pm1}{\alpha_j}^{\pm1}$.
\end{itemize}
\end{theo}

\begin{rem} 
As shown in \cite[Lemma 19]{HarJo}, such a points are contained in a proper subset of the Hecke eigenvariety for $\mathrm{GSp}_{2g}$.
\end{rem}

\begin{defin}\label{defstb}
We say that $\pi_{\pfrak}$ is Steinberg if $\alpha_i= \vert  \NN(\phantom{c}) \vert_{\pfrak}^{i-1}{\alpha_1}$. 
\end{defin} 
If $\pi_{\pfrak}$ is Steinberg at $p$, then $\alpha_{\pfrak,i}(\varpi_{\pfrak}) = q_{\pfrak}^i\alpha_{\pfrak,1}(\varpi_{\pfrak})$. \\
Trivial zeros appears also for automorphic forms which are only partially Steinberg at $\pfrak$  and can be dealt exactly at the same way as the parallel one but for the sake of notation we prefer not to deal with them. \\

To each $g+1$ non-zero elements $(t_1,\ldots,t_g;t_0) \in {(A^{\times})}^{g+1}$ we associate the diagonal matrix \begin{align*}
u(t_1,\ldots,t_g;t_0):=(t_1,\ldots,t_g, t_0 t_g^{-1},\ldots, t_0t_1^{-1})
\end{align*} of $\mathrm{GSp}_{2g}(A)$. \\
For $1 \leq i \leq g-1$ we denote by $u_{\pfrak,i}$ the diagonal matrix associated with $(1,\ldots,1,\varpi_{\pfrak}^{-1},\ldots,\varpi_{\pfrak}^{-1};\varpi_{\pfrak}^{-2})$, where $\varpi_{\pfrak}$ appears $i$ times; we also denote by $u_{\pfrak,0}$ the diagonal matrix corresponding to $(1,\ldots,1;\varpi_{\pfrak}^{-1})$. 
\begin{defin}\label{Upi}
The Hecke operators $U_{\pfrak,i}$, for $1 \leq i \leq g$ are defined as the double coset operator $[\mathrm{Iw}u_{\pfrak,g-i} \mathrm{Iw}]$. \\
\end{defin}

We have that $U_{\pfrak,g}$ is the ``classical'' $U_p$ operator \cite[\S 0]{BS}. 
We shall say then that $\pi$ is of finite slope for $U_{\pfrak,g}$ if $U_{\pfrak,g}$ has eigenvalue $\alpha_{\pfrak,0} \neq 0$ on $\pi_{\pfrak}$.\\

We are interested to study the possible $p$-stabilization of $\pi$ ({\it i.e.} Iwahori fixed vectors). If  $\pi_{\pfrak}$ is unramified at $\pfrak$, we have then $2^g g!$ choices (see \cite[Lemma 16]{HarJo} or \cite[Proposition 9.1]{BS}). If $\pi_{\pfrak}$ is Steinberg, we have instead only one possible choice, as the monodromy $N$ has maximal rank.\\

Suppose that we can lift $\pi$ to an automorphic representation $\pi^{(2^g)}$ of $\mathrm{GL}_{2^g}$. We suppose also that we can lift $\pi$ to an automorphic representation $\pi^{(2g+1)}$ of $\mathrm{GL}_{2g+1}$.\\
Let $V=V_{\mathrm{spin}}$ (resp. $V_{\mathrm{sta}}$) be the Galois representation associated with $\pi^{(2^g)}$ (resp. $\pi^{(2g+1)}$). We make the following assumption
\begin{itemize}
\item[\bfseries{LGp})] $V$ is semistable at all $\pfrak \mid p$ and strong local-global compatibility at $l=p$ holds.
\end{itemize}
These hypotheses are conjectured to be always true for $f$ as above. Arthur's transfer from $\mathrm{GSp}_{2g}$ to $\mathrm{GL}_{2g+1}$ has been proven in \cite{BinXu} (note that it is now unconditional \cite{WaldX}) and for $V=V_{\mathrm{sta}}$ this hypothesis is then verified thanks to \cite[Theorem 1.1]{Cara}. These hypotheses are also satisfied in many cases for $V=V_{\mathrm{spin}}$ in genus $2$ (see \cite{AsgSha,PiSaSc}). \\ 
 Roughly speaking, we require that
\begin{align*}
\mathrm{WD}(V_{\vert F_{\pfrak}})^{\mathrm{ss}} \cong \iota_{n}^{-1} \pi^{(n)}_{\pfrak},
\end{align*}
where $\mathrm{WD}(V_{\vert F_{\pfrak}})$ is the Weil-Deligne representation associated with  $V_{\vert F_{\pfrak}}$ \`a la Berger, $\pi^{(n)}_{\pfrak}$ is the component at $\pfrak$ of $\pi^{(n)}$,  and $\iota_n$ is the local Langlands correspondence for $\mathrm{GL}_n(F_{\pfrak})$ geometrically normalized ($n=2g+1$ when $V$ is the standard representation and $n=2^g$ when $V$ is the spinorial representation).\\
When $\pi_{\pfrak}$ is an irreducible quotient of $\mathrm{Ind}_B^{\mathrm{GSp}_{2g}}(\alpha_{\pfrak,1}\otimes \cdots \otimes \alpha_{\pfrak,g})$ we have that the Frobenius eigenvalues on $\mathrm{WD}({V_{\mathrm{spin}}}_{\vert F_{\pfrak}})^{\mathrm{ss}}$ are the $2^g$ numbers
\begin{align*}
\left(\alpha_{\pfrak,0}  \prod_{ \begin{array}{c}
0 \leq r \leq g \\ 
1 \leq i_1 < \ldots < i_r \leq g 
\end{array}} \alpha_{\pfrak,i_1}(\varpi_{\pfrak}) \cdots \alpha_{\pfrak,i_r}(\varpi_{\pfrak}) \right).
\end{align*}
The ones on $\mathrm{WD}({V_{\mathrm{sta}}}_{\vert F_{\pfrak}})^{\mathrm{ss}}$ are
\begin{align*}
 \left( \alpha_{\pfrak,g}^{-1}(\varpi_{\pfrak}),\ldots,  \alpha_{\pfrak,1}^{-1}(\varpi_{\pfrak}), 1,  \alpha_{\pfrak,1}(\varpi_{\pfrak}), \ldots,  \alpha_{\pfrak,g}(\varpi_{\pfrak}) \right).
 \end{align*} 
Moreover, the monodromy operator should have maximal rank (i.e. one-dimensional kernel) if we are Steinberg or be trivial otherwise. (This is also a consequence of the weight-monodromy conjecture for $V$.)\\
Let $\pfrak$ be a $p$-adic place of $V$ and let $\tau$ be a complex place in $I_{\pfrak}$. The Hodge-Tate weights of ${V_{\mathrm{spin}}}_{\vert_{F_{\pfrak}}}$ at $\tau$ are then 
\begin{align*}
{\left(\frac{k_0}{2}+\frac{1}{2}\sum_{i=1}^g \eps(i)(k_{\tau,i}-i) \right)}_{\eps},
\end{align*} where $\eps$ ranges among the $2^g$ maps from $\lgr 1,\ldots, g\rgr$ to $ \lgr \pm 1 \rgr$. \\
The one of ${V_{\mathrm{sta}}}_{\vert_{F_{\pfrak}}}$ are $(1-k_{\tau,g}, \ldots, g-k_{\tau,1}, 0,k_{\tau,1}-g,\ldots, k_{\tau,g}-1)$. \\
Thanks to work of Tilouine-Urban \cite{TU}, Urban \cite{UrbEig},  Andreatta-Iovita-Pilloni \cite{AIP} we have families of Siegel modular forms;
\begin{theo}\label{teoFami}
Let $\W=\mathrm{Hom}_{\mathrm{cont}}\left(\Z_p^{\times} \times {({(\oo_F \otimes_{\Z} \Z_p)}^{\times})}^{g},\C_p^{\times}\right)$ be the weight space. There exist an affinoid neighborhood $\U$ of $\kappa_0= \left((z,{(z_i)}_{i=1}^g) \mapsto z^{k_0} \prod_{\tau \in I}\prod_i \tau(z_i)^{k_{\tau,i}} \right)$ in $\W$, an equidimensional rigid variety $\X=\X_{\pi}$ of dimension $dg+1$, a finite surjective map $ w : \X  \rightarrow U$,  a character $ \Theta : \h^{Np} \rightarrow  \oo(\X)$, and a point  $x$ in $\X$ above $\underline{k}$ such that $x \circ \Theta$ corresponds to the Hecke eigensystem of $\pi$. \\
Moreover, there exists a dense set of points $x$ of $\X$ coming from classical cuspidal Siegel-Hilbert automorphic forms of weight $(k_{\tau,i};k_0)$ which are regular  and spherical at $p$.
\end{theo}
\begin{rem}
Assuming Leopoldt conjecture, the multiplicative group appearing in the definition of $\W$  is, up to a finite subgroup, ${({(\oo_F \otimes_{\Z} \Z_p)}^{\times})}^{g+1}/\overline{\oo_F^{\times}}$ ({\it i.e.} the $\Z_p$-points of the torus of $\mathrm{Res}_F^{\Q}(\mathrm{GSp}_{2g})$ modulo the $\Z_p$-points of the center).
\end{rem}
This allows us to define two pseudo-representations $R_{?}:G_{\Q} \rightarrow \oo(\X)$, for $?=\mathrm{spin}$, $\mathrm{sta}$,  interpolating the trace of the representations associated with classical Siegel forms \cite[Proposition 7.5.4]{BelCh}. 
Suppose now that $V_?$ is absolutely irreducible (this is conjectured to hold when $\pi$ is Steinberg at least at one prime);  we have then, shrinking $\U$ around $\underline{k}$ if necessary, a {\it big} Galois representation $\rho_{\mathrm{?}}$ with value in $\mathrm{GL}_{n}(\oo(\X))$ such that $\mathrm{Tr}(\rho_?)=R_{?}$ \cite[page 214]{BelCh}. \\

For $1 \leq j < g$ we define $\lambda_{\pfrak}(u_{\pfrak,g-j})=\varpi_{\pfrak}^{\sum_{\tau \in I_{\pfrak}} k_{\tau,1}+ \cdots + k_{\tau,j} - k_0}$ and $\lambda_{\pfrak}(u_{\pfrak,0})=\varpi_{\pfrak}^{\sum_{\tau \in I_{\pfrak}}(k_{\tau,1}+ \cdots + k_{\tau,g} - k_0)/2}$. We have analytic functions $\beta_{\pfrak,j}:=\Theta(U_{\pfrak,j}{\vert \lambda_{\pfrak}(u_{\pfrak,g-j}) \vert_p}) \in  \oo(\X) $. 
We proceed now as in \cite{HarJo}. We recall the following theorem \cite[Theorem 0.3.4]{Liu}; 
\begin{theo}\label{TeoLiu}
Let $\rho:G_{F_{\pfrak}}\rightarrow \mathrm{GL}_n(\oo(\X))$ be a continuous representation. Suppose that there exist $\kappa_1(x),\ldots, \kappa_n(x)$ in $F_{\pfrak}\otimes_{\Q_p}\oo(\X)$, $F_1(x),\ldots,F_d(x)$ in $\oo(\X)$, and a Zariski dense set of points $Z \subset \X$ such that
\begin{itemize}
\item for any $x$ in $\X$, the Hodge-Tate weights of $\rho_x$ are $\kappa_1(x),\ldots, \kappa_n(x)$;
\item for any $z$ in $Z$, $\rho_z$ is crystalline;
\item for any $z$ in $Z$, $\kappa_{\tau,1}(z) < \ldots < \kappa_{\tau,n}(z)$, for all $\tau \in I_{\pfrak}$; 
\item for any $z$ in $Z$, the eigenvalues of $\varphi^{f_{\pfrak}}$ on $\mathcal{D}_{\mathrm{cris}}(V_z)$ are $\prod_{\tau \in I_{\pfrak}}\tau(\varpi_{\pfrak})^{\kappa_{\tau,1}(z)}F_1(z),\ldots,\prod_{\tau \in I_{\pfrak}}\tau(\varpi_{\pfrak})^{\kappa_{\tau,n}(z)}F_n(z)$;
\item for any $C$ in $\R$, defines $Z_C \subset Z$ as the set of points $z$ such that for all $I,J \subset \lgr 1, \ldots,n \rgr$ such that $|\sum_{i\in I} \kappa_{\tau,i}(z) - \sum_{j\in J} \kappa_{\tau,j}(z)|>C$ for all $\tau \in I_{\pfrak}$. We require that for  all $z \in Z$ and $C \in \R$, $Z_C$ accumulates  at  $z$.
\item for $1\leq i \leq n$ there exist $\chi_i : \oo_{F_{\pfrak}}^{\times } \rightarrow {\oo(\X)}^{\times}$  such that $\chi_i(u)=\prod_{\tau} \tau(u)^{\kappa_{\tau, i}(x)}$.
\end{itemize}
Then, for all $x$ in $\X$ non-critical and regular  ($\kappa_1(x) < \ldots < \kappa_n(x)$ and the eigenvalues of $\varphi$ on $\bigwedge^i \mathcal{D}_{\mathrm{cris}}(V_x)$ are distinct for all $i$) there exists a neighborhood $U$ of $x$ such that $\rho_U$ is trianguline and its graded pieces are $\rr_U(\chi_i)$.
\end{theo}
We can apply this theorem and show that the $(\varphi,\G)$-module associated with ${\rho_{\mathrm{?}}}_{\vert G_{\Q_p}}$ is trianguline. We explicit now the triangulation, given in \cite[\S 3.3]{HarJo}.\\

As seen before, a $p$-stabilization of $\pfrak$ corresponds to a permutation $\nu$ and a map $\eps$. 

The eigenvalues of $\varphi$ are given by 
\begin{align*}
\prod_{\tau \in I_{\pfrak}} {\tau(\varpi_{\pfrak})}^{c_1 +\mu_{\tau,1}} {\beta_{\pfrak,1}},\\
\prod_{\tau \in I_{\pfrak}} {\tau(\varpi_{\pfrak})}^{c_i +\mu_{\tau,i}}  \frac{\beta_{\pfrak,i-1}}{\beta_{\pfrak,i}},\\
\prod_{\tau \in I_{\pfrak}} {\tau(\varpi_{\pfrak})}^{c_g +\mu_{\tau,g}}  \frac{\beta_{\pfrak,g-1}}{\beta_{\pfrak,g}^2}
\end{align*}
where $c_i$'s are a positive integer independent of the weight.\\
We define the following characters of $F_{\pfrak}$ with value in $\oo(\X)$: 
\begin{align*}
\chi_{\pfrak,1}(\varpi_{\pfrak})& =  {\beta_{\pfrak,1}},\\
\chi_{\pfrak,i}(\varpi_{\pfrak})& =  \frac{\beta_{\pfrak,i-1}}{\beta_{\pfrak,i}},\\
\chi_{\pfrak,g}(\varpi_{\pfrak})& =  \frac{\beta_{\pfrak,g-1}}{\beta_{\pfrak,g}^2},
\end{align*}
and $\chi_{\pfrak,1}(u)=\prod_{\tau \in I_{\pfrak}} {\tau(u)}^{c_i +\mu_{\tau,i}}$.\\

From \cite[Lemma 19]{HarJo} we have that the graded pieces of $\boldD^{\dagger}_{\mathrm{rig}}({V_{\mathrm{sta}}}_{\vert_{\pfrak}})$ are then given by the characters $\chi_{\pfrak,g},\ldots, 1, \ldots, \chi_{\pfrak,g}^{-1}$.\\

Concerning $V_{\mathrm{spin}}$, we number the subsets of $\lgr 1,\ldots,g \rgr$ as $I_1,I_2,\ldots, I_{2^g}$. Each $I_j$ correspond to a map $\eps_j:\lgr 1,\ldots, g \rgr \rightarrow \pm 1$.\\
We have then the graded pieces $\delta_{\pfrak,j}$ are given by the characters 
\begin{align*}
\delta_{\pfrak,\eps_j}(u)= & \prod_{\tau \in I_{\pfrak}} {\tau(u)}^{d_j +\frac{k_0 +\sum_i \eps_{j}(i)k_{\tau,i}}{2}},\\
\delta_{\pfrak,\eps_j}(\varpi_{\pfrak})=&  \beta_{\pfrak,g} \prod_{i \in I_j} {\chi_{\pfrak,i}}(\varpi_{\pfrak}).
\end{align*}
  
Let $V$ be either $V_{\mathrm{Sta}}$ or $V_{\mathrm{spin}}$.
If $\pi_{\pfrak}$ is Steinberg, there  is only one choice of a regular $(\varphi,N)$-sub-module $D_{\pfrak}$ of $\boldD_{\mathrm{st}}(V_{G_{F_{\pfrak}}})$, where $V$ is one of the two representations associated with $\pi$ described above. 
If the form is not Steinberg at $\pfrak$ many different regular sub-module can be chosen.\\ In any case, we expect (and we shall assume in the follow) that there is at most one trivial zero for each $\pfrak$.  
Consider now the representation $\pi$ of parallel weight $\underline{k}$ ({\it i.e.} associated with $\NN_{F/\Q} (\mathrm{det}^{\underline{k}})$, $\underline{k} \in \Z$) as in the introduction.\\
We give a preliminary proposition on the factorization of the $\Ll$-invariant. Recall the set $S^{\mathrm{Sph,1}}$ and $S^{\mathrm{Stb}}$ defined in the introduction, we have the following;
\begin{prop}
We have the following factorization 
\begin{align*}
\Ll(V,D)=\Ll(V,D)^{\mathrm{Sph}} \prod_{\pfrak \in S^{\mathrm{Stb}}} \Ll(V,D)_{\pfrak},
\end{align*}
where $\Ll(V,D)^{\mathrm{Sph}}$ comes from the prime in $S^{\mathrm{sph}}$ and the factors $\Ll(V,D)_{\pfrak}$ are local.
\end{prop}
\begin{proof}
We follow \cite[\S 1.3]{HTate2}. In the notation of Section \ref{SecLinv}, we write $W_1 = \oplus_{\pfrak \in S^{\mathrm{Stb}}} W_{\pfrak,1}$ and $M_1 = \oplus_{\pfrak \in S^{\mathrm{Sph},1}} M_{\pfrak,1}$. We are left to show that the endomorphism $\iota_{\mathrm{f}} \circ \iota^{-1}_{\mathrm{c}}$ of $\mathcal{D}_{\mathrm{cris}}(W_1\oplus M_1)\cong E^{t}$ keeps stable $\mathcal{D}_{\mathrm{cris}}(M_1)$ and on the quotient it respects the direct sum decomposition $\oplus_{\pfrak \in S^{\mathrm{Stb}}} \mathcal{D}_{\mathrm{cris}}(W_{\pfrak,1})$.\\
Consider a prime $\pfrak_0 \in S^{\mathrm{Stb}}$ and a cocycle $c \in H^1(V,D)$ such that $\mathrm{res}_{\pfrak}(c)=0$ for all $\pfrak \neq \pfrak_0$. This means that $\mathrm{res}_{\pfrak}(c) =0 \in H^1_{\mathrm{f}}(F_{\pfrak},V)=H^1_{\mathrm{f}}(F_{\pfrak},M_{\pfrak}) $ (by (\ref{h1fMV})). Hence by Proposition \ref{Mstruct} (which holds only for $\pfrak$ in $S^{\mathrm{Stb}}$) we have $\iota_{\mathrm{c},\pfrak}(c)= \iota_{\mathrm{f},\pfrak}(c) = 0$. \\
We have also $\iota_{\mathrm{c},\pfrak}(c)=0$ for all primes $\pfrak \neq \pfrak_0$ as $H^1_{\mathrm{c}}$ is the direct sum complement of $H^1_{\mathrm{f}}$ (see \cite[Proposition 1.5.9]{BenLinv}). \\
The proposition then follows from standard linear algebra as in \cite[Corollary 1.9]{HTate2}.
\end{proof}
\begin{rem} A key ingredient in the proof of the factorization at Steinberg places is that each prime ideal brings a single trivial zero.
\end{rem}

We consider now the case $V=V_{\mathrm{sta}}$. We have a contribution to trivial zeros from the $\pi_{\pfrak}$'s which are Steinberg and possibly from the $\pi_{\pfrak}$ which are spherical. In particular, if we choose the regular sub-module coming from an ordinary filtration, we always have a  trivial zero coming from each place. \\
For all $1 \leq s \leq \mathrm{min}(k-g-1,g-1)$ we have also $e_{\mathrm{Stb}}$ trivial zeros for $V(s)$.

\begin{theo}\label{TeoLinv}
For $\pi_{\pfrak}$ Steinberg we have
 \begin{align*}
{\Ll(V,D)}_{\pfrak} = -  \frac{1}{f_{\pfrak}}{\frac{\textup{d} \log_p\beta_{\pfrak,1}(k)}{\textup{d}k}}_{\vert_{k=\underline{k}}},
\end{align*} 
where $k$ is the parallel weight variable.\\
For $ 1 \leq s \leq \mathrm{min}(k-g-1,g-2)$ we also have 
 \begin{align*}
{\Ll(V(s), D(s))}_{\pfrak}= -  \frac{1}{f_{\pfrak}}{\frac{\textup{d} \log_p (\beta_{\pfrak,s-1} \beta_{\pfrak,s}^{-1}(k))} {\textup{d}k}}_{\vert_{k=\underline{k}}}
\end{align*} 
and if $g-1 \leq k-g-1$ we have 
\begin{align*}
{\Ll(V(g-1), D(g-1))}_{\pfrak}=  -  \frac{1}{f_{\pfrak}} {\frac{\textup{d} \log_p (\beta_{\pfrak,g-1}\beta_{\pfrak,g}^{-2}(k))} {\textup{d}k}}_{\vert_{k=\underline{k}}}. 
\end{align*}
\end{theo}
\begin{proof}
We apply Theorem \ref{LinvM} for the $f_{\pfrak}$-th root of $\chi_{\pfrak,i}(\varpi_{\pfrak})$ and we note that we can specialize to a parallel family, so that no contribution from the denominator appears. The $\log_p(u)$ disappear because of the change of variable $T \mapsto u^k -1$ ($u$ any topological generator of $\Z_p^{\times}$).
\end{proof}
\begin{rem}
The presence of $f_{\pfrak}$ in the denominator is explained in term of $L$-function and Euler factors at $p$ in \cite{HTate1}.
\end{rem} 
From now on, $V=V_{\mathrm{spin}}(k-1)$ ($s=k-1$ is the only critical integer); if $\pi_{\pfrak}$ is spherical it should not give any trivial zeros (as the corresponding $p$-adic representation is conjectured to be crystalline and consequently the $\beta_i$'s are Weil numbers of non-zero weight). \\ 
So we are left to see what happen at the primes Steinberg at $\pfrak$. Twisting by $\beta_{\pfrak,g}$ the triangulated $(\varphi,\G)$-module of $\rho_{\mathrm{spin}}$ we are in the hypothesis of Theorem \ref{LinvM} and we have 
\begin{theo}\label{TeoLinv2}
For $\pi_{\pfrak}$ Steinberg we have
 \begin{align*}
{\Ll(V,D)}_{\pfrak} = -  \frac{1}{f_{\pfrak}} {\frac{\textup{d} \log_p \beta_{\pfrak,1}(k)}{\textup{d}k}}_{\vert_{k=\underline{k}}},
\end{align*} 
where $k$ is a parallel weight variable.
\end{theo}


\section{The case of the adjoint representation}\label{Ad}
We prove Theorem \ref{teoAd} of the introduction.  We consider only the case $g=2$. Fix an automorphic representation $\pi$ of weight $\underline{k}=(\underline{k}_{\tau,1},\ldots,\underline{k}_{\tau,g};\underline{k}_{0})_{\tau}$ and let $V=V_{\spin}$ be the spin representation associated with $\pi$. Let $\rho=\rho_{\spin}$ be the corresponding big Galois representation.\\

We specializes the eigenvariety $\X$ of Theorem \ref{teoFami} to the subspace of the weight space given by the equations $k_{\tau,i}=k_{\tau',i}$ if $\tau$ and $\tau'$ induce the same $p$-adic place $\pfrak$ and $k_0=\underline{k}_{0}$. We shall denote the new variable by $k_{\pfrak,i}$ and this eigenvariety by $\X'$. For simplicity, we rewrite the graded pieces of $V$ as 
\begin{align*}
\delta_{\pfrak,1}(\varpi_{\pfrak})=  F_{\pfrak,1}^{-1}(k),&\;\; \delta_{\pfrak,1}(u)= \NN_{F_{\pfrak}/\Q_p}(u)^{\frac{k_0
+ k_{\pfrak,1}+k_{\pfrak,2}-3}{2}}, \\ 
\delta_{\pfrak,2}(\varpi_{\pfrak})=  F_{\pfrak,2}^{-1}(k),&\;\; \delta_{\pfrak,2}(u)= \NN_{F_{\pfrak}/\Q_p}(u)^{\frac{k_0 + k_{\pfrak,2}-k_{\pfrak,1}+1}{2}}, \\
\delta_{\pfrak,3}(\varpi_{\pfrak})= F_{\pfrak,2}(k),&\;\; \delta_{\pfrak,3}(u)= \NN_{F_{\pfrak}/\Q_p}(u)^{\frac{k_0-k_{\pfrak,2}+k_{\pfrak,1}-1}{2}},  \\ 
\delta_{\pfrak,4}(\varpi_{\pfrak})=  F_{\pfrak,1}(k),&\;\; \delta_{\pfrak,4}(u)= \NN_{F_{\pfrak}/\Q_p}(u)^{\frac{k_0 -k_{\pfrak,1}-k_{\pfrak,2}+3}{2}}
\end{align*}
where $k={(k_{\pfrak,1},k_{\pfrak_2};k_0)}_{\pfrak}$. \\
The representation space of $\Ad(V)$ is given by the matrices 
\begin{align*}
\mathfrak{Sp}_4=\lgr X \in \mathfrak{SL}_4 \vert XJ^t + JX=0\rgr.
\end{align*}
The $p$-stabilization on $V$ induces a natural $p$-stabilization and consequently a regular sub-module $D_{\Ad}$ on $\Ad(V_{\spin})$. 
We have 
\begin{align*}
{D_{\Ad}}_{-1} =  & \lgr \mbox{nilpotent }X\rgr, \\
{D_{\Ad}}_0=  & \lgr \mbox{unipotent }X\rgr. 
\end{align*}
The basis for the space ${D_{\Ad}}_0 / {D_{\Ad}}_{-1}$ is given by the two diagonal matrices $d_1=[-1,0,0,1]$ and $d_2=[0,-1,1,0]$. We shall denote by $d_{\pfrak,i}$ these matrices when seen as a vector for $\Ad(V_{\pfrak})$.
\begin{prop}\label{PropImp}
Suppose that {\bfseries{C1-C4}} holds for $V$. Suppose that the classical $E$-point $x$ in the eigenvariety $\X'$ corresponding to $\pi$ is \'etale above the weight space. Then, the space $\Ll(D_{\Ad},V)$ is generated by the image of $\left(\frac{\textup{d}\log_p \delta_{\pfrak,i}} {\textup{d} k_{\pfrak',j}}d_{\pfrak,i}\right)_{\pfrak', j=1,2}$.
\end{prop}
\begin{proof}
The proof is standard and goes back to \cite{MT}, so we shall only sketch it. Let $A=E[T]/(T^2)$. Consider an infinitesimal deformation of $\rho$ given by 
\begin{align*}
\rho_{A} = V \oplus \rho';
\end{align*}
note that $\rho'$ can be written as the first order truncation of $\frac{\partial \rho}{\partial v}$, where $v$ is any direction in the weight space.\\
From $\rho_A$ we can construct a cocyle $c_{x,A}$ defined by \begin{align*} G_F \ni\sigma \mapsto \rho'(\sigma) V^{-1}(\sigma).
\end{align*}
It is easy to check that this defines a cocycle with values in $V \otimes V^*$. Moreover its image lands in $\Ad(V) \subset V \otimes V^*$ as the determinant is fixed (by our choice of the Hodge-Tate weight on $\X'$). Writing explicitly the matrix for the $(\varphi,\G)$-module associated with $\rho_A$ we obtain
$$
{\left(\begin{array}{cccc}
\frac{\partial  \delta_{\pfrak,1}}{\partial v}& * & * & *\\
 & \frac{\partial  \delta_{\pfrak,2}}{\partial v}& * & * \\
 & & \frac{\partial  \delta_{\pfrak,3}}{\partial v}& * \\
 & & & \frac{\partial  \delta{\pfrak,4}}{\partial v} 
\end{array}\right)}_{\vert_{k=\underline{k}}}
{\left(\begin{array}{cccc}
  \delta_{\pfrak,1}^{-1}& * & * & *\\
 &  \delta_{\pfrak,2}^{-2}& * & * \\
 & &  \delta_{\pfrak,3}^{-1}& * \\
 & & &  \delta_{\pfrak,4}^{-1} 
\end{array} \right)}_{\vert_{k=\underline{k}}}
$$
In particular, they are upper triangular and their projection via $\iota_{\mathrm{f}}$ onto the vector $d_{\pfrak,1}$ is ${\frac{\textup{d}\log_p F_{\pfrak,1}(k)}{\textup{d} v}}_{\vert_{k=\underline{k}}} $. Similarly for $d_{\pfrak,2}$. \\
We also have that the projection via $\iota_{\mathrm{c}}$  onto $d_{\pfrak,1}$ is $-{\frac{\partial (k_{\pfrak,1}+k_{\pfrak,2})/2}{\partial v}}_{\vert_{k=\underline{k}}}$. \\
This cocycle lies $H^1(G_{F,S},V\otimes V^*)$ by construction of $\rho$.\\
As $\lgr \frac{\partial}{\partial k_{\pfrak,i}} \rgr_{\pfrak, i=1,2}$ is a base of the tangent space at $x$ in $\X'$ we are done.
\end{proof}
We can now prove Theorem \ref{teoAd} which we recall now;
\begin{theo}
We have 
\begin{align*}
\Ll(\mathrm{Ad}(V_{\spin}),D_{\Ad}) =\prod_{\pfrak} \frac{2}{f^2_{\pfrak}} {\mathrm{det}{\left(\begin{array}{cc}
\frac{\partial \log_p F_{\pfrak_i,1}(k)}{\partial k_{\pfrak_j,1}} & \frac{\partial \log_p F_{\pfrak_i,2}(k)}{\partial k_{\pfrak_j,1}} \\
\frac{\partial \log_p F_{\pfrak_i,1}(k)}{\partial k_{\pfrak_j,2}} & \frac{\partial \log_p F_{\pfrak_i,2}(k)}{\partial k_{\pfrak_j,2}}
\end{array}  \right)}_{1 \leq i,j \leq t}}_{\vert_{k=\underline{k}}}.
\end{align*}
\end{theo}
\begin{proof}
Once we have Proposition \ref{PropImp}, we just have to follow the proof of \cite[Theorem 3.73]{HIwa}. The matrix of $\iota_{\mathrm{f}}$ is exactly what appears in the Theorem, while the matrix of $\iota_{\mathrm{c}}$ can be directly calculated using the formula $\frac{\textup{d}\log(u^{\pm k_{\pfrak,i}})}{\textup{d} k_{\pfrak',j}}= \pm \delta_{\pfrak,\pfrak'}\delta_{i,j}$ (where $\delta_{a,b}$ here is Kronecker delta) and gives a contribution of $2^{-1}$ for each prime ideal $\pfrak$.
\end{proof}
\bibliographystyle{alpha}
\bibliography{Bibliografy} 
\Addresses
\end{document}